\documentclass[12pt,twoside]{article}
	\usepackage[utf8]{inputenc}
	\usepackage{a4wide}
	\usepackage{amsfonts}
	\usepackage{amsmath}
	\usepackage{amssymb}
	\usepackage{amsthm}
	\usepackage{mathrsfs}
	\usepackage{enumerate}
	\usepackage{tikz}	
	\usetikzlibrary{matrix,arrows}

	\newcommand{\ZZ}{\mathbb{Z}}

	\newcommand{\Oo}{\mathscr{O}}

	\DeclareMathOperator{\chara}{char}

	\DeclareMathOperator{\Pic}{Pic}
	\DeclareMathOperator{\Cl}{Cl}
	\DeclareMathOperator{\Div}{Div}

	\newcommand{\isom}{\simeq}

	\renewcommand{\phi}{\varphi}
	\renewcommand{\epsilon}{\varepsilon}
	\renewcommand{\tilde}{\widetilde}

	\newtheorem*{theorem*}{Theorem}

	\title{A note on the Frobenius morphism on toric varieties}
	\author{Piotr Achinger}
\begin{document}

\maketitle
\begin{abstract}
We give a new, shorter computation of Frobenius push-forwards of line bundles on
toric varieties. 
\end{abstract}

Frobenius push-forwards of line bundles on smooth toric varieties were calculated
by Thomsen \cite{Thomsen} -- it was observed that they split into direct sums
of line bundles. A more intrinsic approach has been developed in \cite{Bogvad}.
The algorithm of Thomsen has been adapted to the case of the toric Frobenius morphism
in characteristic zero in \cite{CostaMiroRoig} and \cite{MichalekLason}.
In this note, we present a new proof of these results. This gives
the shortest known computation even for the projective spaces (in which case
one uses the Horrocks splitting criterion and the projection formula to prove
that the direct image splits). The key point of our approach is to consider
Frobenius push-forwards of all line bundles at once.

Let $X$ be a toric variety over an algebraically closed field $k$. 
We denote the torus acting on $X$ by $T$ and the number of rays of the fan
defining $X$ by $r$. If
$\chara k = p > 0$, we have the \emph{honest} (absolute) Frobenius morphism $F_a:X\to X$.
In fact, $X$ can be identified with its Frobenius twist $X^{(1)}$ and then $F_a$
can be seen as the quotient by the Frobenius kernel $K_a$ of $T$.
In any case, for any integer $\ell>0$ we have the \emph{fake} (toric) Frobenius
morphism $F_\ell : X\to X$ which corresponds to taking the quotient by the kernel
$K_\ell$ of the $\ell$-th power map on $T$. Let $F$ be the honest or a fake Frobenius
(in the honest case we define $\ell = p$)
in the following theorem:

\begin{theorem*} Let $D\in\Cl X$. Then
\begin{equation}\label{dec}
	 F_* \Oo_X(D) = \bigoplus_{E \in \Cl X} \Oo_X(E)^{\oplus m(E, D)}, 
\end{equation}
where the multiplicity $m(E, D)$ equals
the number of points in the cube $\{0, 1, \ldots, \ell-1\}^r\subseteq \ZZ^r = \Div_T X$
representing the class $D-\ell E\in\Cl X$ (that is, the number of 
$T$-divisors in $|D-\ell E|$ with coefficients less than $\ell$).
\end{theorem*}

\begin{proof}

Let us first prove the theorem in the case when $X$ is smooth and complete, and
reduce to this case afterwards. 

First of all, we remark that the push-forward $F_* \Oo_X(D)$ of a $T$-equivariant
line bundle $\Oo_X(D)$ is a direct sum of line bundles. Indeed, the 
kernel $K$ of $F$ on $T$ (equal to $K_a$ or $K_\ell$) is a finite diagonalizable commutative group scheme
acting on $F_* \Oo_X(D)$ and the eigensheaves are line
bundles. Since every line bundle on $X$ is equivariant, we get a decomposition
as in (\ref{dec}) and we only want to compute the multiplicities.

Observe that $m(E, D)$ depends only on $D - \ell E$: by the projection
formula we have $(F_* \Oo_X(D))\otimes \Oo_X(-E) = F_*(\Oo_X(D - \ell E))$,
so $m(E, D) = m(0, D-\ell E)$. Denote $m(0, D)$ simply by $m(D)$ and 
apply $h^0(-)$ to both sides of (\ref{dec}):
\begin{equation} \label{h0}
	 h^0(D) = h^0(F_* \Oo_X(D)) = \sum_{E \in \Pic X} m(E, D)\cdot h^0(E) = \sum_{E \in \Pic X} m(D - \ell E)\cdot h^0(E). 
\end{equation}

We want to use some generating functions, so we fix a basis $D_1, \ldots, D_\rho$
of $\Pic X$ such that the effective cone lies in the positive orthant and define
\[ S(x_1, \ldots, x_\rho) = \sum_{a\in \ZZ^\rho} h^0\left(\sum a_i D_i\right) x^a \quad\text{and}\quad M(x_1, \ldots, x_n) = \sum_{a\in \ZZ^\rho} m\left(\sum a_i D_i\right) x^a. \]
Then (\ref{h0}) just states that $S(x_1, \ldots, x_\rho) = M(x_1, \ldots, x_\rho)\cdot S(x_1^\ell, \ldots, x_\rho^\ell)$.

Let us compute the series $S$. Consider the map $L : \ZZ^r = \Div_T X \to \Pic X = \ZZ^\rho$
(the first identification being given by the basis of ,,ray'' divisors, the second by $D_1, \ldots, D_\rho$)
taking a $T$-divisor to its class.
Because $h^0(D)$ equals the number of effective $T$-divisors linearly equivalent to $D$, we get
\[ S(x_1, \ldots, x_\rho) = \sum_{b\in \ZZ_{\geq 0}^r} x^{L(b)} = \prod_{i=1}^r \frac{1}{1 - x^{L(e_i)}}, \]
$e_1, \ldots, e_r$ being the basis in $\ZZ^r$. Therefore
\[ M(x_1, \ldots, x_\rho) = \frac{S(x_1, \ldots, x_\rho)}{S(x_1^\ell,\ldots, x_\rho^\ell)} = \prod_{i=1}^r \frac{1-x^{\ell L(e_i)}}{1-x^{L(e_i)}}
 = \prod_{i=1}^r (1 + x^{L(e_i)} + x^{2L(e_i)} +\ldots + x^{(\ell -1)L(e_i)}), \]
hence $m(D)$ is the number of points $p = \sum a_i e_i$ with
$0 \leq a_i < \ell$ and $L(p) = D$.

\medskip
We turn to the case $X$ not necessarily smooth nor complete. First, by taking
an appropriate subdivision of the fan we get a toric resolution of singularities
$\pi: \tilde X \to X$. Next, by adding extra cones,
we embed $i: \tilde X\to \bar X$ into a smooth complete toric variety $\bar X$.
Using the diagrams
\[ 
\begin{tikzpicture}[description/.style={fill=white,inner sep=2pt}]
    \matrix (m) [matrix of math nodes, row sep=1.0em,
    column sep=2.5em, text height=1.5ex, text depth=0.25ex]
    {      & \ZZ^{\bar r - \tilde r} & \ZZ^{\bar r - \tilde r} \\
	M  & \Div_T \bar X           & \Pic \bar X \\
	M  & \Div_T \tilde X           & \Pic \tilde X \\ };
    \path[->,font=\scriptsize]
    (m-1-2) edge node[auto] {$ \isom $} (m-1-3);
    \path[right hook->,font=\scriptsize]
    (m-2-1) edge node[auto] {$ $} (m-2-2);
    \path[->>,font=\scriptsize]
    (m-2-2) edge node[auto] {$ $} (m-2-3);
    \path[right hook->,font=\scriptsize]
    (m-3-1) edge node[auto] {$ $} (m-3-2);
    \path[->>,font=\scriptsize]
    (m-3-2) edge node[auto] {$ $} (m-3-3);
    \path[->,font=\scriptsize]
    (m-2-1) edge node[auto] {$ \isom $} (m-3-1);
    \path[right hook->,font=\scriptsize]
    (m-1-2) edge node[auto] {$ $} (m-2-2);
    \path[->>,font=\scriptsize]
    (m-2-2) edge node[auto] {$ $} (m-3-2);
    \path[right hook->,font=\scriptsize]
    (m-1-3) edge node[auto] {$ $} (m-2-3);
    \path[->>,font=\scriptsize]
    (m-2-3) edge node[auto] {$ $} (m-3-3);
\end{tikzpicture}
\quad
\begin{tikzpicture}[description/.style={fill=white,inner sep=2pt}]
    \matrix (m) [matrix of math nodes, row sep=1.0em,
    column sep=2.5em, text height=1.5ex, text depth=0.25ex]
    {   \,\\ 
	\text{and} \\
	\,\\  };
\end{tikzpicture}
\quad
\begin{tikzpicture}[description/.style={fill=white,inner sep=2pt}]
    \matrix (m) [matrix of math nodes, row sep=1.0em,
    column sep=2.5em, text height=1.5ex, text depth=0.25ex]
    {      & \ZZ^{\tilde r - r} & \ZZ^{\tilde r - r} \\
	M  & \Div_T \tilde X           & \Pic \tilde X \\
	M  & \Div_T X           & \Cl X \\ };
    \path[->,font=\scriptsize]
    (m-1-2) edge node[auto] {$ \isom $} (m-1-3);
    \path[right hook->,font=\scriptsize]
    (m-2-1) edge node[auto] {$ $} (m-2-2);
    \path[->>,font=\scriptsize]
    (m-2-2) edge node[auto] {$ $} (m-2-3);
    \path[right hook->,font=\scriptsize]
    (m-3-1) edge node[auto] {$ $} (m-3-2);
    \path[->>,font=\scriptsize]
    (m-3-2) edge node[auto] {$ $} (m-3-3);
    \path[->,font=\scriptsize]
    (m-3-1) edge node[auto] {$ \isom $} (m-2-1);
    \path[->>,font=\scriptsize]
    (m-2-2) edge node[auto] {$ $} (m-1-2);
    \path[right hook->,font=\scriptsize]
    (m-3-2) edge node[auto] {$ $} (m-2-2);
    \path[->>,font=\scriptsize]
    (m-2-3) edge node[auto] {$ $} (m-1-3);
    \path[right hook->,font=\scriptsize]
    (m-3-3) edge node[auto] {$ $} (m-2-3);
\end{tikzpicture}
\]
it is easy to see that if the theorem holds for $\bar X$, then it also holds
for $\tilde X$ and for $X$ (the variety $X$ being normal).
\end{proof}

\noindent {\bf Acknowledgements. } 
The author would like to thank Nathan Ilten, Mateusz Michałek, Nicolas Perrin
and Jarosław Wisniewski for valuable suggestions. This work was supported by the Hausdorff Center for
Mathematics.

\medskip
\noindent {\sc Piotr Achinger \\
Hausdorff Center for Mathematics,
Universit{\"a}t Bonn \\
Villa Maria \\
Endenicher Allee 62\\
53115 Bonn, Germany} \\

\noindent {\sc E-mail:} \texttt{piotr.achinger@hcm.uni-bonn.de}

\end{document}